\documentclass{amsart}
\usepackage{amssymb}
\usepackage{mathrsfs}
\usepackage{graphicx}
\usepackage{comment}
\newtheorem{Prop}{Proposition}[section]
\newcommand{\ve}{\varepsilon}
\newcommand{\nni}{\noindent}

\def\S{\mathcal S}

\newcommand{\be}{\begin{equation}}
\newcommand{\ee}{\end{equation}}
\newcommand{\ba}{\begin{align}}
\newcommand{\ea}{\end{align}}

\newcommand{\abs}[1]{\lvert#1\rvert}

\newtheorem{example}{Example}[section]
\newtheorem{theorem}{Theorem}[section]
\newtheorem{corollary}[theorem]{Corollary}
\newtheorem{lemma}{Lemma}[section]
\newtheorem{alg}{Algorithm}[section]
\newtheorem{question}{Question}

\newenvironment{iolist}[1]%
{\begin{list}{}{%
\settowidth{\labelwidth}{\textsf{{\it #1.}}}%
\setlength{\labelsep}{4mm}%
\setlength{\leftmargin}{\labelwidth}%
\addtolength{\leftmargin}{\labelsep}%
}}%
{\end{list}}

\def\beq{\begin{equation}}\def\enq{\end{equation}}

\newenvironment{biglabellist}[1]%
{\begin{list}{}{%
\settowidth{\labelwidth}{\textsf{{\it #1.}}}%
\setlength{\labelsep}{2mm}%
\setlength{\leftmargin}{\labelwidth}%
\addtolength{\leftmargin}{\labelsep}%
\addtolength{\leftmargin}{4mm}%
\setlength{\itemsep}{6pt}%
\setlength{\listparindent}{0pt}%
\setlength{\topsep}{3pt}%
}}%

\title[group determinants]{The group determinants for $\mathbb Z_n \times H$}

\author[B. Paudel]{Bishnu Paudel}
\address{ Department of Mathematics\\
         Kansas State University\\
         Manhattan, KS 66506, USA}
\email{bpaudel@ksu.edu, pinner@math.ksu.edu}

\author[C. Pinner]{Christopher Pinner}

\date{\today}
\begin{document}

\begin{abstract}
Let $\mathbb Z_n$ denote the cyclic group of order $n$.
We show how the group determinant for  $G= \mathbb Z_n \times H$ can be simply written in  terms of the group determinant for $H$.
We use this to get a complete description of the integer group determinants for $\mathbb Z_2 \times D_8$ where $D_8$ is the dihedral group of order 8, and $\mathbb Z_2 \times Q_8$ where $Q_8$ is the quaternion group of order 8.

\end{abstract}

\maketitle

\section{Introduction} 

At the meeting of the American Mathematical Society in Hayward, California, in April 1977, Olga Taussky-Todd \cite{TausskyTodd} asked whether one could characterize the values of the group determinant when the entries are all integers.
There was particular interest in the case of $\mathbb Z_n$, the cyclic group of order $n$, where the group determinant corresponds
to the $n\times n$ circulant determinant. For a prime $p,$ a complete description was obtained for  the cyclic groups $\mathbb Z_{p}$ and $\mathbb Z_{2p}$ in \cite{Newman1} and \cite{Laquer}, and for $D_{2p}$ and $D_{4p}$ in \cite{dihedral}. Here $D_{2n}$ denotes the dihedral group of order $2n$. In general though this quickly becomes a hard problem,
with only partial results known even for $\mathbb Z_{p^2}$ once $p\geq 7$ (see \cite{Newman2} and \cite{Mike}).
A complete description has  though been  obtained for all groups of order less than 16  (see \cite{smallgps} and \cite{bishnu1}),
and for 6 of the 14 groups  of order 16,  $D_{16}$  and  the five  abelian groups $\mathbb Z_{16}$, $\mathbb Z_2 \times \mathbb Z_8$, $\mathbb Z_2^4$, $\mathbb Z_4^2$ and $\mathbb Z_2^2 \times \mathbb Z_4$ (see  \cite{dihedral,Yamaguchi1,Yamaguchi2,Yamaguchi3,Yamaguchi4} and \cite{Yamaguchi5}). We write $\mathcal{S}(G)$ for the set of integer group determinants for the group $G$. 

Our goal here is to show how the group determinant for a group of the form $G=\mathbb Z_n \times H$ can be straightforwardly related to the group determinants for the group $H$. We use this to give a complete description for two more non-abelian groups of 
order 16, namely $\mathbb Z_2\times D_8$ and  $\mathbb Z_2\times Q_8$ where $Q_8$ is the quaternion group.

Here we shall think of the  group determinants as being defined on elements of the group ring $\mathbb C [G]$
$$ \mathcal{D}_G\left( \sum_{g\in G} a_g g \right)=\det\left( a_{gh^{-1}}\right) ,$$
although our ultimate interest is of course  in the integer group determinants $\mathbb Z [G]$. We observe  the multiplicative property
\be \label{mult} \mathcal{ D}_G(xy)= \mathcal{D}_G(x)\mathcal{D}_G(y), \ee
using that
$$ x=\sum_{g \in G} a_g g,\;\;\;  y=\sum_{g \in G} b_g g \; \Rightarrow  \; xy=\sum_{g\in G} \left(\sum_{hk=g}a_hb_k\right) g. $$

\section{Products with $\mathbb Z_n$}

We show that when $G=\mathbb Z_n \times H$ we can write our integer group $G$-determinant as  a product
of $n$ group $H$-determinants of elements in $\mathbb Z[\omega_n][H],$ where $\omega_n:=e^{2\pi i/n}.$

\begin{theorem} If $G=\mathbb Z_n \times H$ then for any $a_{ih}$ in $\mathbb C$
\be \label{product}  \mathcal{D}_G\left( \sum_{i=0}^{n-1}\sum_{h\in H} a_{ih}(i,h)\right)=\prod_{y^n=1} \mathcal{D}_H\left( \sum_{h\in H} \left(\sum_{i=0}^{n-1} a_{ih} y^i\right) h\right). \ee
\end{theorem}
Results of this flavour have been obtained before \cite{Yamaguchi0}, but here we do not need to assume that $H$ is abelian.

\begin{proof}
One way to see this is to use Frobenius' factorisation \cite{Frob} of the group determinant in terms of the irreducible, non-isomorphic, representations  $\rho$ of $G$ (see for example \cite{Conrad} or \cite{book})
$$ \mathcal{D}_G\left(\sum_{g\in G} a_g g\right) =\prod_{\rho} \det\left(\sum_{g\in G} a_g\rho(g) \right)^{\deg(\rho)}.$$
Observe that every representation $\rho$ for $H$ extends to $n$ representations for G
$$ \rho_y ( i,h) = y^i \rho(h), $$
where $y$ runs through the $n$th roots of unity.

More directly we can alternatively follow Newman's proof \cite{Newman1} of the factorization of the group determinant for $G=\mathbb Z_n$.
Newman observes that the group matrix $M$ for $\sum_{i\in \mathbb Z_n} A_i i,$ that is the circulant matrix with first row
$A_0, A_1,\ldots , A_{n-1},$ takes the form 
$$ M=A_0I_n+A_1P+\cdots +A_{n-1}P^{n-1}, \;\;\; P=\begin{pmatrix} 0 & 1 & 0 & \cdots & 0 \\ 0 & 0 & 1 & \cdots &  0 \\
  &  &  & \cdots   & \\
0 & 0 & 0 & \cdots & 1 \\
1 &  0 & 0 & \cdots & 0 \end{pmatrix}.$$ Now  $P$ has eigenvalues $y$, $y^n=1$, so the matrix $M$  will
have the same eigenvectors  as $P$ but with eigenvalues 
\be \label{ev} A_0+A_1y+\cdots A_{n-1}y^{n-1},\;\;y=1,\omega_n,\omega_n^2,\ldots ,\omega_{n}^{n-1}.\ee
Hence the matrix of eigenvectors $B$
will yield a diagonal matrix $B^{-1}MB$ with the values 
\eqref{ev} down the diagonal.

 Now suppose that $H=\{h_1,\ldots ,h_m\}$ and order the elements so that the first row of the $G=\mathbb Z_n\times H$
group matrix ${\mathcal  M}$ for $\sum_{i\in \mathbb Z_n, h\in H} a_{ih}(i,h)$ consists of the 
$$a_{0h_1},\ldots ,a_{0h_m},a_{1h_1},\ldots ,a_{1h_m},\ldots, a_{(n-1)h_1},\ldots ,a_{(n-1)h_m}. $$
Then it is not hard to see that first $m$ rows of our $G$ group matrix ${\mathcal M}$ will consists of  $m\times m$ blocks $A_0,A_1,\ldots, A_{n-1},$
where $A_i$ is the group $H$ matrix associated to $\sum_{h\in H} a_{ih}h,$ and the subsequent rows the same blocks cyclically permuted.

Hence if we take the $n\times n$ matrix $B$ and replace each entry $a_{ij}$ with the $m\times m$ block $a_{ij}I_m$
we obtain an $nm \times nm$ matrix ${\mathcal B},$ where ${\mathcal B}^{-1} {\mathcal M {\mathcal B}}$ will now be a block matrix with entries the same linear
combinations of the blocks $A_i$ as occured for the elements in $B^{-1}MB$; that is blocks \eqref{ev} down the diagonal
and zeros elsewhere. The result is then plain.
\end{proof}

Notice that if we start with an integer $G$ group determinant, then we can assemble the $n$ determinants in
\eqref{product} into   $\tau (n)$ integers by combining the primitive $d$th roots of unity, $d\mid n$. If $H$ is abelian, then  these will be  integer $H$ group determinants
\be \label{combine} \prod_{\stackrel{y=\omega_d^j}{\gcd(j,d)=1}} \mathcal{D}_H\left( \sum_{h\in H} \left(\sum_{i=0}^{n-1} a_{ih} y^i\right) h\right)= \mathcal{D}_H\left( \prod_{\stackrel{y=\omega_d^j}{\gcd(j,d)=1}}\sum_{h\in H} \left(\sum_{i=0}^{n-1} a_{ih} y^i\right) h\right), \ee
since the resulting coefficients will be symmetric expressions in the conjugates and hence in $\mathbb Z$.
In particular, an integer $G=\mathbb Z_n \times H$ group determinant is an integer group $H$ determinant, though this can 
be seen more directly  (if $H=\mathbb Z_{n_1}\times \cdots \times \mathbb Z_{n_k},$ then the $G$ group determinant reduces to a product of an integer polynomial $F(y,x_1,\ldots ,x_k)$ over the $n,n_1,\ldots ,n_k$th roots of unity and $\prod_{y^n=1}F(y,x_1,\ldots ,x_k)$ is just an integer polynomial in one less variable; see also \cite[Theorem 1.4]{Yamaguchi3}).
If $H$ is nonabelian,  then the process \eqref{combine} may leave elements in
$\mathbb Z[\omega_d] [H]$, and we are unable to say that an integer $G$ group determinant must be  an integer $H$ group determinant,  except for the case when $n=2$.




\section{ The group $ \mathbb Z_2 \times D_8$}

Notice that when $n=2$ we can write an integer $\mathbb Z_2 \times H$ group determinant as a  product of two
integer $H$ group determinants:
$$ \mathcal{D}_{\mathbb Z_2 \times H}\left( \sum_{h\in H} a_h(0,h) + \sum_{h\in H} b_h (1,h)\right) = \mathcal{D}_H \left( \sum_{h\in H} (a_h+b_h)h\right) \mathcal{D}_H\left( \sum_{h\in H} (a_h-b_b) h \right). $$

In the case of $H=D_8=\langle F,R\: | \; F^2=1,R^4=1, RF=FR^3\rangle$  we take the coefficients of the group elements
 $(0,R^j)$, $(1,R^j)$, $(0,FR^j)$ and $(1,FR^j)$, as the coefficients of $x^j$ 
in four cubics, $f_1,f_2,g_1$ and $g_2$ respectively. The $\mathbb Z_2\times D_8$ determinant, which we will
denote
$\mathcal{D}(f_1,f_2,g_1,g_2)$, is then the product of two $D_8$ determinants, which from \cite{dihedral} can be written
\be \label{defD}\mathcal{D}(f_1,f_2,g_1,g_2)=\mathcal{D}(1) \mathcal{D}(-1),\;\;\;\mathcal{D}(z)=m_1(z)m_2(z)\ell(z)^2, \ee
where
\begin{align*}  m_1(z) & =(f_1(1)+zf_2(1))^2-(g_1(1)+zg_2(1))^2, \\
m_2(z) & =(f_1(-1)+zf_2(-1))^2-(g_1(-1)+zg_2(-1))^2,
\end{align*}
and
\be \label{defell}\ell (z)= |f_1(i)+zf_2(i)|^2-|g_1(i)+zg_2(i)|^2. \ee

We obtain a complete description of the $\mathbb Z_2 \times D_8$ integer group determinants.

\begin{theorem}\label{MainD8}
For $G=\mathbb Z_2 \times D_8$ the set of  odd integer group determinants is
$$ A:=\{ m(m+16k) \; : \; m,k\in \mathbb Z,\; m \text{ odd} \}. $$
The even determinants are the $2^{16}m$, $m\in \mathbb Z.$

\end{theorem}
Notice  the set of achieved odd values $A$ consists of  all the integers 1 mod 16 and exactly those integers $9$ mod 16 which contain a prime $p\equiv \pm 3$ or $\pm 5$ mod 16. These are the same as the odd values found in  \cite{Yamaguchi3} for $\mathbb Z_2 \times \mathbb Z_8$. In fact $\mathcal{S}(\mathbb Z_2 \times D_8)\subsetneq \mathcal{S}(\mathbb Z_2 \times \mathbb Z_8).$
Sets of this type occur for other 2-groups.

\begin{Prop}\label{Prop1}
The odd integer group determinants for  $G=\mathbb Z_2 \times \mathbb Z_{2^n}$  are the
$$ \{ m(m+|G|k) \; : \; m,k\in \mathbb Z,\; m \text{ odd }\}. $$
The odd  integer group determinants for  $G=\mathbb Z_2^n$ or $\mathbb Z_2^n\times \mathbb Z_4$ are the $m\equiv 1$ mod $|G|$.
\end{Prop}
Note, for any group the $m\equiv 1$ mod $|G|$ are in $\mathcal{S}(G)$, with $1+k|G|$ obtained  by taking $a_g=1+k$ for the identity and $a_g=k$ for the others.

\begin{proof}[Proof of Theorem \ref{MainD8}]

{\bf Achieving the values.} We write $H(x):=(x+1)(x^2+1)$. 

We achieve the values in $A$ with $m=\pm 1,\pm 9$ mod 16 from 
$$\mathcal{D}(1+kH,kH,kH,kH)=1+16k. $$ 
We get those with $\pm m\equiv 5$ mod 16 from $(5+16t)(5+16k)$ achieved with
$$\mathcal{D}(1+x+x^2+(t+k)H,x-x^3+(t-k)H,1+x+(t+k)H,1-x^2+(t-k)H)$$
and those with $\pm m\equiv 3$ mod 16 from $(3+16t)(3+16k)$  achieved with
$$\mathcal{D}(1+x+(t+k)H,1+x-x^2-x^3+(t-k)H,1+x-x^3+(t+k)H,x-x^3+(t-k)H).$$
We achieve the even values $2^{18}m$ using
$$\mathcal{D}(1+x+x^2-mH,1-x^2-x^3-mH,1+x-x^3+mH,x+mH),$$
the $2^{17}(2m+1)$  with
$$\mathcal{ D}(1+x+x^2+x^3+mH,1+x+mH,x+mH,1+x^2-x^3+mH),$$
the $2^{16}(1+4m)$ from
$$ \mathcal{D}(1+x+x^2+x^3+mH,1+x-x^2-x^3+mH,1+x-x^2-x^3+mH,1-x+mH),$$
and the $2^{16}(-1+4m)$ from
$$ \mathcal{D}(1+x+x^2-mH,1+x-x^3-mH,1-x^3-mH,x-x^2-mH). $$

\noindent
{\bf The odd values.} We show that any odd determinant must lie in $A$. We know that any $\mathbb Z_2 \times D_8$ determinant must be the product of two $D_8$
determinants, which we can write
\be \label{GenForm1} \mathcal{D}_1=m_1m_2\ell_1^2, \;\;\;  \mathcal{D}_2=m_3m_4\ell_2^2 \ee
with 
$$ m_1=f(1)^2-g(1)^2,\; \;\; m_2=f(-1)^2-g(-1)^2,\; \;\;\ell_1=|f(i)|^2-|g(i)|^2, $$
and
\begin{align*}  m_3 & =(f(1)+2h(1))^2-(g(1)+2k(1))^2,  \\
m_4 & =(f(-1)+2h(-1))^2-(g(-1)+2k(-1))^2, \\
\ell_2 & = |f(i)+2h(i)|^2-|g(i)+2k(i)|^2.
\end{align*}
Assume that  $\mathcal{D}_1\mathcal{D}_2$ is odd. Switching $f$ and $g$ and replacing $f$ by $-f$ as necessary, we shall  assume that $f(1)\equiv 1$ mod 4 and $2\mid g(1)$.
The result will follow once we show that 
$$ \mathcal{D}_1\equiv \mathcal{D}_2 \pmod{ 16}. $$
We write
$$ h(x)=\sum_{i=0}^3 a_i(x-1)^i,\;\; k(x)=\sum_{i=0}^3 b_i(x-1)^i, \;\;  f(x)=\sum_{i=0}^3 c_i(x-1)^i,\;\; g(x)=\sum_{i=0}^3 d_i(x-1)^i,$$
where $c_0=1$ mod 4 and $2\mid d_0$.
Now 
\begin{align*} m_3-m_1 & =4f(1)h(1)+4h(1)^2- 4k(1)g(1)-4k(1)^2\equiv 4a_0+4a_0^2-4b_0d_0-4b_0^2 \mod 16,\\
m_4-m_2& =4f(-1)h(-1)+4h(-1)^2- 4k(-1)g(-1)-4k(-1)^2 \\ 
& \equiv 4(a_0-2a_1-2a_0c_1)+4a_0^2 -4(b_0d_0-2 b_0d_1) -4b_0^2   \mod 16 ,
\end{align*}
and
\begin{align*} \ell_2-\ell_1 & = (2h(i)\overline{f(i)}+2\overline{h(i)}f(i)+4|h(i)|^2)-  (2k(i)\overline{g(i)}+2\overline{k(i)}g(i)+4|k(i)|^2)\\
\equiv &  (4a_0-4a_0c_1-4a_1+4a_0^2)-(   -4b_0d_1    + 4b_0^2)     \mod 8 \end{align*}
and
$$ \ell_2^2 \equiv \ell_1^2 +   8 (a_0-a_0c_1-a_1+a_0^2+b_0d_1    -b_0^2) \mod 16. $$
Since $m_1,m_2,\ell_1^2\equiv 1$ mod 4 we get 
\begin{align*}  \mathcal{D}_1-\mathcal{D}_2  \equiv   &   4(a_0+a_0^2-b_0d_0-b_0^2)+ 4(a_0-2a_1-2a_0c_1+a_0^2 -b_0d_0+2 b_0d_1-b_0^2)   \\ & +  8 (a_0-a_0c_1-a_1+a_0^2+b_0d_1    -b_0^2) \equiv 0 \mod 16.
\end{align*}

\vspace{1ex}
\noindent
{\bf The even values.} We know from \cite{dihedral} that the even $D_8$ determinants are divisible by $2^8$. So any even $\mathbb Z_2 \times D_8$ determinant $\mathcal{D}_1\mathcal{D}_2$ must be a multiple of $2^{16}$, and all these are achieved.
\end{proof}

\begin{proof}[Proof of Proposition \ref{Prop1}] Suppose that  $H$ is an abelian 2-group and $G=\mathbb Z_2\times H$. Then by
\cite[Theorem 2.3]{dilum2}  we can write the $G$-determinant as a product of two $H$-determinants $\mathcal{D}=\mathcal{D}_1\mathcal{D}_2$ with $\mathcal{D}_2\equiv \mathcal{D}_1$ mod $|G|$, and
\be \label{strong} \mathcal{S}(\mathbb Z_2 \times H) \subseteq \{m(m+k|G|)\; : \; m\in \mathcal{S}(H)\}.\ee
For  $G=\mathbb Z_2\times \mathbb Z_{t}$, $t=2^n$, the determinants take the form $\mathcal{D}_1=\prod_{x^t=1} F(x,1),$ $\mathcal{D}_2=\prod_{x^t=1} F(x,-1)$ for some
$F(x,y)=f(x)+yh(x)$, the coefficients of $x^i$ in $f$ and $h$ corresponding to the $a_g$ for $g=(0,i)$ and $(1,i)$ respectively. For an odd positive integer  $m,$ taking 
$$F(x,y) = \prod_{p^{\alpha}\parallel m} \left(\frac{x^p-1}{x-1}\right)^{\alpha}+k(y+1)\left(\frac{x^t-1}{x-1}\right)$$ 
achieves $m(m+k|G|)$.
 
For $G=\mathbb Z_2^n$ all  the odd values must be $1$ mod $|G|$ by \cite[Lemma 2.1]{dilum1}.  For $G_n=\mathbb Z_2^n\times \mathbb Z_4$ observe when $n=1$ that $m(m+8k)\equiv m^2\equiv 1$ mod 8, and in general
that if  $m\equiv 1$ mod $|G_{n-1}|$ then $m(m+|G_n|k)\equiv m^2 \equiv 1$ mod $|G_n|.$
\end{proof}
Interestingly \eqref{strong} also holds for the non-abelian groups $H=D_8$ and $Q_8$.

\section{The group $\mathbb Z_2 \times Q_8$}

A $\mathbb Z_2\times Q_8$ determinant will be a product of two $Q_8$ determinants, which by \cite{smallgps}
can be written in a very similar way to \eqref{defD};
\be \label{defQ} \mathcal{D}(f_1,f_2,g_1,g_2)=\mathcal{D}(1) \mathcal{D}(-1),\;\;\; \mathcal{D}(z)=m_1(z)m_2(z)\ell(z)^2, \ee
with
\begin{align*}  m_1(z) & =(f_1(1)+zf_2(1))^2-(g_1(1)+zg_2(1))^2, \\
m_2(z) & =(f_1(-1)+zf_2(-1))^2-(g_1(-1)+zg_2(-1))^2,
\end{align*}
but now
$$ \ell(z)=|f_1(i)+zf_2(i)|^2+|g_1(i)+zg_2(i)|^2. $$
Writing $Q_8=\langle A,B : A^4=1, B^2=A^2, AB=BA^{-1}\rangle,$ the coefficient
of $x^i$  in  the cubic  $f_1,f_2,g_1$ and $g_2,$  corresponds to the $a_g$ in the $\mathbb Z_2 \times Q_8$ group determinant  for $g=(0,A^i),(1,A^i),(0,BA^i)$ and $(1,BA^i)$ respectively.

We obtain a complete description of the $\mathbb Z_2 \times Q_8$ integer group determinants.

\begin{theorem} When $G=\mathbb Z_2 \times Q_8$ the odd integer group determinants are the integers $1$ mod 16, plus the integers  $9$ mod 16 of the form
$$ s_1s_2 (\ell_1\ell_2)^2,\;\; s_1,s_2\equiv -3 \mod 8,   \;\; \ell_1,\ell_2\equiv 3 \mod 4, $$
for some $s_1,s_2$ in $\mathbb Z$ and $\ell_1,\ell_2$ in $\mathbb N,$ with
\be \label{res1}  s_1\equiv s_2 \mod 16\;\;  \text{ and } \;\;  \ell_1\equiv \ell_2 \mod 8 \ee
or
\be \label{res2}  s_1\equiv s_2+8 \mod 16 \;\; \text{ and }\;\;  \ell_1\equiv \ell_2 +4\mod 8. \ee
The even values are the 
$2^{18}m,$ $m$ in $\mathbb Z$, the
$$ 2^{17}(2m+1)p^2,\;\;\; m\in \mathbb Z, \;\;p\equiv 3 \mod 4, $$
the $2^{16}m$ with $m\equiv 1,3$ or $5$ mod 8, and those with $m\equiv 7$ mod 8 of the form
\be \label{16type1}  2^{16}(8t-1)\ell^2, \;\; t\in \mathbb Z,\;\; \ell\in \mathbb N,\; \ell\equiv 1 \text{ mod } 4,\; \ell\geq 5, \ee
or  
\be \label{16type2}  (8t+3)(8k-3) 2^{16},\;\; t,k\in \mathbb Z. \ee
\end{theorem}
For the three non-abelian groups of order 16 whose $\mathcal{S}(G)$ is now known we have:
$$ \mathcal{S}(\mathbb Z_2 \times Q_8)\subsetneq \mathcal{S}(\mathbb Z_2 \times D_8)\subsetneq \mathcal{S}(D_{16})=\{4m+1 \; :\; m\in \mathbb Z\} \cup \{ 2^{10}m \; : \;  m\in \mathbb Z\}. $$


\vskip0.2in
\nni
\begin{proof} {\bf Achieving the odd values.} We write $H(x):=(x+1)(x^2+1)$. 

We achieve the values 1 mod 16 from 
$$\mathcal{D}(1+kH,kH,kH,kH)=1+16k. $$ 
To achieve the specified values 9 mod 16 we write the $\ell_i$ as a sum of 4 squares. Notice that since they are 3 mod 4 
we must have 3 of them, ${A}_i,{B}_i,{C}_i$ say, odd and one, ${D}_i$, even, with $2\parallel {D}_i$ if $\ell \equiv 7$ mod 8 and $4\mid {D}_i$
if $\ell\equiv 3$ mod 8. Hence with a choice of sign we can write
$$ \ell_i ={A}_i^2 +{B}_i^2+{C}_i^2+{D}_i^2,   \;\;\; {A}_i,{B}_i,{C}_i \text{ odd, }\;\;  {D}_i  \text{ even},  $$
with
$$ A_1\equiv A_2 \mod 4,\;\;\;  B_1\equiv B_2 \mod 4,\;\; \;C_1\equiv C_2 \mod 4,$$
and 
$$D_1 \equiv D_2 \mod 4 \;\;\; \text{ if } \;\;\; \ell_1\equiv \ell_2 \mod 8, $$
and 
$$ D_1 \equiv D_2 +2\pmod 4 \;\;\; \text{ if } \;\;\;  \ell_1\equiv \ell_2 +4  \pmod 8. $$

In the case $\ell_1\equiv \ell_2$ mod 8 we get $(8m-3)(8k-3)(\ell_1\ell_2)^2$, $m\equiv k$ mod 2, from
\begin{small}
\begin{align*}
f_1 & =\frac{1}{4}(D_1+D_2) + \frac{1}{4}(B_1+B_2-2) x -\frac{1}{4}(D_1+D_2)x^2 -\frac{1}{4}(B_1+B_2+2)x^3+\frac{1}{2}(m+k)H,\\
f_2 & =\frac{1}{4}(D_1-D_2) + \frac{1}{4}(B_1-B_2) x -\frac{1}{4}(D_1-D_2)x^2 -\frac{1}{4}(B_1-B_2)x^3+\frac{1}{2}(m-k)H,\\
g_1&=\frac{1}{4}(C_1+C_2-2) + \frac{1}{4}(A_1+A_2-2) x -\frac{1}{4}(C_1+C_2+2)x^2 -\frac{1}{4}(A_1+A_2+2)x^3+\frac{1}{2}(m+k)H,\\
g_2& = \frac{1}{4}(C_1-C_2) + \frac{1}{4}(A_1-A_2) x -\frac{1}{4}(C_1-C_2)x^2 -\frac{1}{4}(A_1-A_2)x^3+\frac{1}{2}(m-k)H.
\end{align*}
\end{small}
In the case $\ell_1\equiv \ell_2 +4$ mod 8 we get $(8m-3)(5-8k)(\ell_1\ell_2)^2$, $m\equiv k$ mod 2, from
\begin{small}
\begin{align*}
f_1 & =\frac{1}{4}(D_1+D_2-2) + \frac{1}{4}(B_1+B_2-2) x -\frac{1}{4}(D_1+D_2+2)x^2 -\frac{1}{4}(B_1+B_2+2)x^3+\frac{1}{2}(m+k)H,\\
f_2 & =\frac{1}{4}(D_1-D_2+2) + \frac{1}{4}(B_1-B_2) x -\frac{1}{4}(D_1-D_2-2)x^2 -\frac{1}{4}(B_1-B_2)x^3+\frac{1}{2}(m-k)H,\\
g_1&=\frac{1}{4}(C_1+C_2-2) + \frac{1}{4}(A_1+A_2-2) x -\frac{1}{4}(C_1+C_2+2)x^2 -\frac{1}{4}(A_1+A_2+2)x^3+\frac{1}{2}(m+k)H,\\
g_2& = \frac{1}{4}(C_1-C_2) + \frac{1}{4}(A_1-A_2) x -\frac{1}{4}(C_1-C_2)x^2 -\frac{1}{4}(A_1-A_2)x^3+\frac{1}{2}(m-k)H.
\end{align*}
\end{small}

\vskip0.01in
\nni
{\bf Odd values must be of the stated form.} We proceed as in the case of $\mathbb Z_2 \times D_8,$ with
\eqref{GenForm1} becoming
\be \label{GenForm2} \mathcal{D}_1=m_1m_2\ell_1^2, \;\;\;  \mathcal{D}_2=m_3m_4\ell_2^2, \ee
where again
$$ m_1=f(1)^2-g(1)^2,\; \;\; m_2=f(-1)^2-g(-1)^2, $$
\begin{align*}  m_3 & =(f(1)+2h(1))^2-(g(1)+2k(1))^2,  \\
m_4 & =(f(-1)+2h(-1))^2-(g(-1)+2k(-1))^2,
\end{align*}
but this time
$$ \ell_1=|f(i)|^2+|g(i)|^2,\;\; \ell_2=|f(i)+2h(i)|^2+|g(i)+2k(i)|^2. $$
This does not change $\ell_1-\ell_2$ mod 8, so again we must have $\mathcal{D}_1\equiv \mathcal{D}_2$ mod 16
and $\mathcal{D}_1\mathcal{D}_2$ is 1 or 9 mod 16. The 1 mod 16 are all achievable, so assume $\mathcal{D}_1\mathcal{D}_2\equiv 9$ mod 16.
Since all the $m_i\equiv m_1$ mod 4, plainly the $\mathcal{D}_i\equiv m_1^2\ell_i^2\equiv 1$ mod 4, so we can assume that $\mathcal{D}_1,\mathcal{D}_2\equiv -3$ mod 8. Since the $\ell_i^2\equiv 1$ mod 8,
we get $m_1m_2,m_3m_4\equiv -3$ mod 8.
Since $m_1$ and $m_2$ are 1 or $-3$ mod 8 we must have one of each, and $2\parallel g(1)$ and $4\mid g(-1)$ or vice versa.
Hence $d_1$ must be odd and 
$$ \ell_1 \equiv (c_0+c_1)^2+c_1^2 +  (d_0+d_1)^2+d_1^2 \equiv 3 \mod 4. $$
From above we also  know 
that  $\ell_1\equiv \ell_2$ mod 4. 
That is, $\mathcal{D}_1\mathcal{D}_2 =s_1s_2 (\ell_1\ell_2)^2$ with $s_1\equiv s_2 \equiv -3$ mod 8 and $\ell_1\equiv \ell_2\equiv 3$ mod 4.
Plainly we have $s_1s_2\equiv 9$ mod 16 if $s_1\equiv s_2\equiv -3$ or 5  mod 16 and $s_1s_2\equiv 1$ mod 16  if one is $-3$ and the other 5 mod 16, while $\ell_1\ell_2\equiv 1$ mod 8 and $(\ell_1\ell_2)^2 \equiv 1 $ mod 16 if $\ell_1\equiv \ell_2\equiv 3$ or $7$ mod 8
and $\ell_1\ell_2\equiv -3$ mod 8  and $(\ell_1\ell_2)^2 \equiv 9 $ mod 16 if one  is $3$ and the other $7$ mod 8.
Hence the restrictions \eqref{res1} and \eqref{res2}  to get $\mathcal{D}_1\mathcal{D}_2\equiv 9$ mod 16.

\vskip0.2in
\nni
{\bf Achieving the  even values}. 
We obtain the $2^{18}m$, $m$ odd, from
\begin{align*}
f_1& =1+x^2 +\frac{1}{2}(m+1) H,\hskip 0.2in
f_2=\frac{1}{2}(m-1) H,\\
g_1& =-(1+x)+\frac{1}{2}(m+1) H,\hskip0.2in
g_2 =\frac{1}{2}(m-1) H,
\end{align*}
and the $2^{19}m$ from
\begin{align*}
f_1&= 1+x+x^2-mH,\hskip0.2in
f_2 = -x-x^3-mH,\\
g_1&= x+x^3+mH,\hskip0.2in
g_2= -x^3+mH. 
\end{align*}
We get the $2^{16}(4m+1)$ from
\begin{align*}
f_1 &=1+x+x^2+x^3+mH,\hskip0.2in
f_2=mH, \\
g_1& =1-x +mH,\hskip0.2in
g_2=mH.
\end{align*}
We get $2^{16}(8t+3)(4s+1)$ from
\begin{align*}
f_1& = 1+x+x^2+x^3+(t+s)H,\hskip0.2in 
f_2 =1+x^2-x^3+(t-s)H,\\
g_1& = (t-s)H,\hskip0.2in
g_2=x^3+ (t+s)H,
\end{align*}
with $s=0$ giving us the $2^{16}m$, $m\equiv 3$ mod 8, and $s=2k-1$ the values \eqref{16type2}.
For $\ell \geq 5$ with $\ell\equiv 1$ mod 4 we can write  $2\ell-4\equiv 6$ mod 8 as a sum of three squares with two of 
them odd and the other 2 mod 4:
$$ 2\ell = (4a+1)^2+2^2+(4c-1)^2 +(4d-2)^2 $$
and we can get $2^{16}(4m-1)\ell^2$, and hence \eqref{16type1}, from
\begin{align*}
f_1 & = (1-x+x^2)+ a(1-x^2)+mH,\\
f_2 &=-x(1+x)+a(1-x^2)+mH,\\
g_1& =-x+(1-x^2)(c+dx)+ mH,\\
g_2& =-(1+x) +(1-x^2)(c+dx)+mH. 
\end{align*}
For $p\equiv 3$ mod 4 we write $2p=A^2+B^2+C^2+D^2$ where, since $2p\equiv 6$ mod 8, two of $A,B,C,D$  must be odd and two even, with one of them divisible by 4, the other 2 mod 4.  Changing signs as necessary we assume that $A=1+4a$, $B=4b$, $C=1+4c$, and
$D=2+4d$. We achieve $2^{17}(2m+1)p^2$ with
\begin{align*}
f_1 & = (1+x)(x^2+1)+a(1-x^2)+bx(1-x^2)+mH,\\
f_2 &=1+(x-1)(x^2+1)+a(1-x^2)+bx(1-x^2)+mH,\\
g_1& =1+x+c(1-x^2)+dx(1-x^2)+ mH,\\
g_2& =x +c(1-x^2)+dx(1-x^2)+mH. 
\end{align*}
\vskip0.01in
\nni
{\bf Even values must be of the stated form.}  We know if the $\mathbb Z_2 \times Q_8$ determinant is even, then both $Q_8$ determinants
are even, and by \cite{smallgps} must each be multiples of $2^8$. Hence the even determinants must be multiples of $2^{16}$.
Note, if the determinant is even we must have $f(1)$ and $g(1)$ the same parity, and all the terms $m_1,m_2,m_3,m_4,\ell_1,\ell_2$ in \eqref{GenForm2} must be even.

\vskip0.1in
\nni
{\bf  The $2^{17}\parallel \mathcal{D}$ are of the stated form.} Suppose that we had a determinant $2^{17}m$, $m$ odd, with $m$ not divisible by the square of a prime 3 mod 4.
Writing 
\be \label{ldiff} \ell_2-\ell_1 \equiv 4(a_0c_0-a_0c_1-a_1c_0+a_0^2)+ 4(b_0d_0-b_0d_1-b_1d_0+b_0^2) \mod 8 \ee
we see that $2\parallel \ell_1,\ell_2$ or $4\mid \ell_1,\ell_2$. If $f(1),g(1)$ are both odd then $2^3\mid m_1,m_2,m_3,m_4$
and we must have $2\parallel \ell_1,\ell_2$ (else $2^{12+8}\mid \mathcal{D}$). Now if $2^{u}\parallel f(1)$ and $2^{v}\parallel g(1)$ with $u,v\geq 1$ then $2^{2\min\{u,v\}}\parallel m_1$ if $u\neq v$, 
while if $u=v$ we have $2^{2u+3}\mid m_1$. Likewise for $m_2,m_3,m_4$.  To obtain an odd power of two
we must therefore have at least one of the $m_i$ with $u=v$. We can not have two of them (else $2^{5+5+4+4}\mid \mathcal{D}$). Again we can assume that $2\parallel \ell_1,\ell_2$
(otherwise $2^{5+6+8}\mid \mathcal{D}$). Since $\ell_1$ and $\ell_2$ do not contain any primes $3$ mod 4 we have $\ell_1\equiv \ell_2 \equiv 2$ mod 8 and \eqref{ldiff} gives
\be \label{diffcong} a_0c_0-a_0c_1-a_1c_0+a_0^2+ b_0d_0-b_0d_1-b_1d_0+b_0^2\equiv 0 \mod 2. \ee

Suppose first that $f(1)$, $g(1)$ are odd. Since $c_0$ and $d_0$ are odd,  \eqref{diffcong} becomes
\be \label{cong1} -a_0c_1-b_0d_1 \equiv a_1+b_1  \mod 2. \ee
 To get power 17, rearranging if necessary to make the highest power on $m_1,$ we must have $2^4\parallel m_1$, $2^3\parallel m_2,m_3,m_4.$ That is, $m_1\equiv 0 \mod 16$, and $m_2\equiv m_3\equiv m_4\equiv 8 \mod 16.$
From 
$$m_1=c_0^2- d_0^2\equiv 0 \mod 16,\;\;\; m_3=(c_0+2a_0)^2- (d_0+2b_0)^2\equiv 8 \mod 16 $$
we get
\be \label{cong2}    a_0 c_0+a_0^2-b_0^2-b_0d_0 \equiv 2 \mod 4.  \ee
From
\begin{align*} m_4  & \equiv (c_0-2c_1+4c_2+2a_0-4a_1)^2 -( d_0-2d_1+4d_2 +2b_0-4b_1)^2  \mod 16 \\
 & \equiv m_2+4a_0^2 +4(a_0c_0-2a_0c_1-2a_1c_0) -4b_0^2 -4(d_0b_0-2d_1b_0-2d_0b_1) \mod 16,
\end{align*}
we get
$$ a_0^2+a_0c_0-b_0^2-b_0d_0-2(a_0c_1+a_1  -d_1b_0-b_1)\equiv 0 \mod 4. $$
Applying \eqref{cong1}, this becomes  $a_0^2+a_0c_0-b_0^2-b_0d_0\equiv 0 \mod 4,$ contradicting \eqref{cong2}.

Now suppose that $2^u\parallel f(1), g(1)$, $u\geq 1$. Since $c_0$ and $d_0$ are even, \eqref{diffcong} becomes
\be \label{cong3} -a_0c_1+a_0^2-b_0d_1+b_0^2 \equiv 0 \mod 2.     \ee
Notice we can't have $c_1,d_1$ both odd or both even, else
$$\ell_1= \abs{c_0-c_1+ic_1 +2\alpha}^2+  \abs{d_0-d_1+id_1 +2\beta}^2\equiv 2c_1^2+2d_1^2 \mod 4 $$
would be divisible by 4.

We can't have $a_0,b_0$ both odd, else \eqref{cong3} becomes $c_1+d_1\equiv 0$ mod 2, contradicting $c_1,d_1$ having opposite parity.
If $u=1$ we can rule out $a_0,b_0$ both even, else $2\parallel f(1)+2h(1)=c_0+2a_0$, $g(1)+2k(1)=d_0+2b_0$ 
(we ruled out $m_1$ and $m_3$ both having $u=v$). 
If $u\geq 2$ we can't have $a_0,b_0$ both even, else 4 divides both terms, $2^4\mid m_3$, $2^7\mid m_1$ and $2^{7+2+4+2+4}\mid D$.  So $a_0,b_0$ like $c_1,d_1$ have opposite parity.
From
$$ f(-1)+2h(-1) \equiv c_0-2c_1 +2a_0 \mod 4,\;\;\; g(-1)+2k(-1)=d_0-2d_1+2b_0 \mod 4 $$
we can't have $a_0\equiv c_1$ mod 2 and $b_0\equiv d_1$ mod 2, else if $u=1$ we would have a single 2 dividing both
(ruled out) and if $u=2$ we would have 4 dividing both and $2^4\mid m_4$, $2^7\mid m_1.$
Hence we must have $a_0\equiv d_1$, $b_0\equiv c_1$ mod 2 and \eqref{cong3} becomes $a_0^2+b_0^2\equiv 0$ mod 2,
contradicting that $a_0,b_0$ have opposite parity.

\vskip0.1in
\nni
{\bf  The $2^{16}\parallel \mathcal{D}$ are of the stated form.} Suppose now that we have $\mathcal{D}=2^{16}m,$ with $m\equiv -1$ mod 8, that is not of the form 
\eqref{16type1} or \eqref{16type2}. Note, $\ell_1\ell_2$ does not contain
a prime $p\equiv 1$ mod 4 or two primes $p_1,p_2\equiv 3$ mod 4 (else it will be type \eqref{16type1}  with $\ell=p$ or $p_1p_2$),
and $\mathcal{D}$ has no factor $\pm 3$ mod 8 (else it will be type \eqref{16type2}).

If $2^2\mid \ell_1$ or $\ell_2$ then $2^2\parallel \ell_1,\ell_2$ and $2^2\parallel m_1,m_2,m_3,m_4$ and $f(1),g(1)$ are
even. Now $m_1/4=(f(1)/2)^2-(g(1)/2)^2\equiv \pm 1$ mod 8 and likewise for $m_2/4,m_3/4$ and $m_4/4$, with their 
product $-1$ mod 8. Switching $f$ and $g$ as necessary and rearranging, we can assume $m_1/4\equiv -1$ mod 8
and $m_2/4,m_3/4,m_4/4\equiv 1$ mod 8. That is $4\mid f(1)/2$ and $f(1)/2+h(1),f(-1)/2,f(-1)/2+h(-1)$ are all odd.
From the first two $h(1)$ is odd, from the  second  two, $h(-1)$ is even, but $h(1)$ and $h(-1)$ must have the same parity.
Hence we can assume that $2\parallel \ell_1,\ell_2$, moreover that $\ell_1=\ell_2=2,$ or one is 2 and the other $2p$ 
for some prime $p=3$ mod 4.

 If $f(1)=c_0$ and $g(1)=d_0$ are odd, then plainly $2^3\parallel m_1,m_2,m_3,m_4$. We rule out one of $\ell_1$, $\ell_2$
being  2  mod 8 and the other 6 mod 8. In this case \eqref{ldiff} becomes
\be \label{diff6odd}    1\equiv -a_0c_1 -a_1-b_0d_1-b_1 \mod 2   \ee
But  the difference of 
$$m_4\equiv (c_0-2c_1+4c_2+2a_0-4a_1)^2-(d_0-2d_1+4d_2+2b_0-4b_1)^2 \equiv 8 \mod 16 $$
and
$$m_2\equiv (c_0-2c_1+4c_2)^2-(d_0-2d_1+4d_2)^2 \equiv 8 \mod 16 $$
gives
$$ 4(a_0^2+a_0c_0-b_0^2-b_0d_0) -8(a_1c_0+a_0c_1-b_1d_0-b_0d_1) \equiv 0 \mod 16, $$
where
$$  4(a_0^2+a_0c_0-b_0^2-b_0d_0)=m_3-m_1 \equiv 0 \mod 16,$$
and $a_1+a_0c_1-b_1-b_0d_1\equiv 0$ mod 2, contradicting \eqref{diff6odd}. This just leaves us with the case $\ell_1=\ell_2=2$
considered in the lemma below.

Suppose $f(1)=c_0=2c$, $g(1)=d_0=2d$ are even. If $c$ and $d$ have opposite parity then $2^2\parallel m_1$, if both are odd
then $2^5\mid m_1$ and if $c=2c'',d=2d''$ then $2^4\parallel m_1$ if $c''$ and $d''$ have opposite parity and $2^6\mid m_1$ otherwise. Moreover if $c$ and $d$ have the same parity and $2^2$  divides
$$   m_1/4=c^2-d^2, $$
then $2^4$ must also divide at least one of the other $m_i$. To see this observe that if $a_0$ and $b_0$ have the same parity then $2^2$ divides
\be \label{m3} m_3/4=(c+a_0)^2-(d+b_0)^2, \ee
if $c_1$ and $d_1$ have the same parity then $2^2$ divides
\be \label{m2}  m_2/4 \equiv (c-c_1+2c_2)^2 - (d-d_1+2d_2)^2 \mod 8, \ee
and if  both $a_0$ and $b_0,$ and $c_1$ and $d_1$ have opposite parity, then $a_0-c_1$ and $b_0-d_1$ have the same parity and $2^2$ divides
\be \label{m4} m_4/4\equiv (c-c_1+2c_2+a_0-2a_1)^2-(d-d_1+2d_2+b_0-2b_1)^2 \mod 8. \ee
Hence, rearranging as necessary, we can assume that $2^4\parallel m_1$ and one other $m_i$, and $2^2\parallel m_i$ for the other two $m_i$. In particular $c_0=4c'$ and $d_0=4d'$ with $c'$, $d'$ of opposite parity. Suppose now that one of
$\ell_1$, $\ell_2$ is $2$ mod 8 and the other 6 mod 8, so that \eqref{ldiff} becomes
\be \label{diff2} 1\equiv -a_0c_1+a_0^2-b_0d_1+b_0^2 \mod 2. \ee
Notice that this rules out $a_0,b_0$ both even or $c_1,d_1$  both odd. We can rule out $a_0,b_0$ both odd, else $2^3\mid m_3/4$ in \eqref{m3},
and $c_1,d_1$ both even else 
$$ 4\mid \ell_1 =\abs{4c'+c_1(i-1)+2i\alpha}^2+\abs{4d'+d_1(i-1)+2i\beta}^2. $$
Hence $a_0-c_1$ and $b_0-d_1$ have the same parity, but can't be odd, else $2^3\mid m_4/4$ in \eqref{m4}. Hence
$c_1\equiv a_0$ mod 2 and $d_1\equiv b_0$ mod 2, violating \eqref{diff2}. This just leaves the case $\ell_1=\ell_2=2$
dealt with in the next lemma.
\end{proof}

\begin{lemma} All $\mathbb Z_2 \times Q_8$ determinants $2^{16}m$ with $m\equiv 7$ mod 8 and  $\ell_1=\ell_2=2$ in \eqref{GenForm2} must be of the form \eqref{16type2}.
\end{lemma}

\begin{proof} Suppose that $\mathcal{D}=2^{16}m$, where $\ell_1=\ell_2=2,$ and all the factors of $m$ are $\pm 1$ mod 8. We show that $m\equiv 1$ mod 8. Hence any with $m\equiv -1$ mod 8 must have a factor $\pm 3$ mod 8 and be of the form  \eqref{16type2}. 

\vskip0.1in
\nni
{\bf Case 1: $f(1)$ and $g(1)$ are even}. Since $\abs{f(i)}^2 $ and $\abs{g(i)}^2$ are both even, we must have one of them 2 and the other 0. Switching $f$ and $g$ and replacing $f(x)$ by $\pm f(\pm x)$ as necessary, we can assume that $f(i)=1+i$ and $g(i)=0.$ Hence we can write
$$f(x)=1+x+(x^2+1)v(x), \;\;\; g(x)=(x^2+1)u(x). $$
 Note $2^2$ divides $\abs{g(i)+2k(i)}^2,$ so this term must also be zero,
while $f(i)+2h(i)=\ve+\delta  i$ with $\delta,\ve=\pm 1,$ and  
$$f(x)+2h(x)=\ve+\delta x + (x^2+1)(v(x)+2h_1(x)),\;\;\; g(x)+2k(x)=(x^2+1)(u(x)+2k_1(x)). $$
Hence
\begin{align*}
m_1/4 & = (1+v(1))^2 -u(1)^2,\;\;\;\;\;  m_2/4=v(-1)^2-u(-1)^2,\\
m_3/4 &= \left(\frac{1}{2}(\ve +\delta) +v(1)+2h_1(1)\right)^2 - (u(1)+2k_1(1))^2,\\
m_4/4 &= \left(\frac{1}{2}(\ve -\delta) +v(-1)+2h_1(-1)\right)^2 - (u(-1)+2k_1(-1))^2.
\end{align*}
Note, one of $1+v(1)$ and $v(-1)$ must be odd, and hence $u(1),u(-1)$ must be even (else $2^3\mid m_1/4$ or $m_2/4$).

\vskip0.1in
\noindent
{\bf (i) Suppose that $v(1)$ is odd.} Since $m_2/4\not\equiv -3$ mod 8 we must have $m_2/4\equiv 1$ mod 8, $4\mid u(-1)$, 
and
$$ \frac{m_1}{2^4}=\left(\frac{1+v(1)}{2}\right)^2 - \left(\frac{u(1)}{2}\right)^2. $$
If $\delta =-\ve$ then $m_3/4 \equiv 1$ mod 8, $4\mid  (u(1)+2k_1(1)),$ and
$$ \frac{m_4}{2^4}=\left( \frac{\ve + v(-1)}{2} +h_1(-1)\right)^2-\left(\frac{u(-1)}{2}+k_1(-1)\right)^2. $$
If $2\mid u(1)/2$ then $2\mid k_1(1),$ $2\mid (u(-1)/2+k_1(-1)),$ and $m_1/2^4,m_4/2^4\equiv 1$ mod 8.
If $2\nmid u(1)/2$  then $2\nmid k_1(1)$,  $2\nmid (u(-1)/2+k_1(-1)),$ and $m_1/2^4,m_4/2^4\equiv -1$ mod 8.
In both cases $m_1m_2m_3m_4/2^{12}\equiv 1$ mod 8.

If $\delta =\ve$  we have $m_4/4\equiv 1$ mod 8, $4\mid (u(-1)+2k_1(-1)),$ $2\mid k_1(-1),k_1(1)$ and 
$$ \frac{m_3}{2^4}=\left( \frac{\ve + v(1)}{2} +h_1(1)\right)^2-\left(\frac{u(1)}{2}+k_1(1)\right)^2. $$
If $2\mid u(1)/2$ then $m_1/2^4,m_3/2^4\equiv 1$ mod 8 and if $2\nmid u(1)/2$ both are $-1$ mod 8. Again 
$m_1m_2m_3m_4/2^{12}\equiv 1$ mod 8.

\vskip0.1in
\noindent
{\bf (ii) Suppose that $v(1)$ is even.} In this case $m_1/4\equiv 1$ mod 8, $4\mid u(1)$ and
$$ \frac{m_2}{2^4}= \left(   \frac{v(-1)}{2}\right)^2 - \left(\frac{u(-1)}{2}\right)^2. $$
If $\delta =-\ve$ then $m_4/4 \equiv 1$ mod 8, $4\mid  (u(-1)+2k_1(-1)),$ and
$$ \frac{m_3}{2^4}=\left( \frac{v(1)}{2} +h_1(1)\right)^2-\left(\frac{u(1)}{2}+k_1(1)\right)^2. $$
If $2\mid u(-1)/2$ then $m_2/2^4\equiv 1$ mod 8 and $2\mid k_1(-1)$, $2\mid (u(1)/2+k_1(1))$ and $m_3/2^4\equiv 1$ mod 8.  If $2\nmid u(-1)/2$ then $m_2/2^4\equiv -1$ mod 8, $2\nmid k_1(-1)$, $2\nmid (u(1)/2+k_1(1))$ and $m_3/2^4\equiv -1$ mod 8. Again, $m_1m_2m_3m_4/2^{12}\equiv 1$ mod 8.

If $\delta =\ve$ then $m_3/4 \equiv 1$ mod 8, $4\mid  (u(1)+2k_1(1)),$ $2\mid k_1(1), k_1(-1)$ and
$$ \frac{m_4}{2^4}=\left( \frac{v(-1)}{2} +h_1(-1)\right)^2-\left(\frac{u(-1)}{2}+k_1(-1)\right)^2. $$
If $2\mid u(-1)/2$ then $m_2/2^4,m_4/2^4\equiv 1$ mod 8.  If $2\nmid u(-1)/2$ then $m_2/2^4,m_4/2^4\equiv -1$ mod 8.  In both cases $m_1m_2m_3m_4/2^{12}\equiv 1$ mod 8.

In conclusion,  there are no cases where $\mathcal{D}/2^{16}=m_1m_2m_3m_4/2^{12}\equiv -1$ mod 8.

\vskip0.2in
\nni
{\bf Case 2: $f(1)$ and $g(1)$ are odd}. In this case we have $2^3\parallel m_1,m_2,m_3,m_4$. From $\ell_1=\ell_2=2$
we must have  $f(i),g(i),f(i)+2h(i),g(i)+2k(i)=\pm 1$ or $\pm i$. Multiplying the $f$  and $h$ or the  $g$ and $k$ through
by $\pm 1$ or $\pm x$ we can assume that $f(i)=1$ and $g(i)=1$ and 
$$f(x)=1+(x^2+1)v(x),\;\; g(x)=1+(x^2+1)u(x). $$
Clearly we must have  $f(i)+2h(i),g(i)+2k(i)=\pm 1$ and
$$ f(x)+2h(x)=\ve+(x^2+1)(v(x)+2h_1(x)),\;\;  g(x)+2k(x)=\delta+(x^2+1)(u(x)+2k_1(x)) $$
for some $\ve,\delta=\pm 1.$ Hence
\begin{align*} \frac{m_1}{4} =\left(1+u(1)+v(1)\right) & (v(1)-u(1)), \\
\frac{m_3}{4}  = \left( \frac{\ve+\delta}{2} +v(1)+u(1) +2h_1(1)+2k_1(1)\right) &\left(\frac{\ve-\delta}{2} +v(1)-u(1) +2h_1(1)-2k_1(1)\right).\end{align*}
Similarly for $m_2/4$ and $m_4/4$ with $u(-1),v(-1),h_1(-1),k_1(-1)$ in place of $u(1),v(1),h_1(1),k_1(1).$ 

\vskip0.1in
\noindent
{\bf (i) Suppose that $u(1)+v(1)$ is even.}  In this case $m_1/8=\alpha_1\alpha_2$ with
$$  \alpha_1=1+u(1)+v(1),\;\;\; \alpha_2=\frac{1}{2}(v(1)-u(1))  $$
When  $\delta =-\ve$ we have $m_3/8=\lambda_1\lambda_2,$ with
\begin{align*}
\lambda_1 & =   \frac{1}{2}(v(1)+u(1)) +h_1(1)+k_1(1),\\
\lambda_2& =\ve +v(1)-u(1)+2h_1(1)-2k_1(1). 
\end{align*}
Recall that by assumption all these factors are $\pm 1$ mod 8. 
Since
$$ \lambda_1= \ve \alpha_2 +\frac{1}{2}(1-\ve)v(1)+\frac{1}{2}(1+\ve) u(1) +h_1(1)+k_1(1), $$
we have $2\mid \frac{1}{2}(1-\ve)v(1)+\frac{1}{2}(1+\ve) u(1) +h_1(1)+k_1(1)$
and
$$ \lambda_2= \ve \alpha_1 +(1-\ve)v(1)-(1+\ve)u(1)+2h_1(1)-2k_1(1) \equiv \ve \alpha_1 \mod 4. $$
Hence $\lambda_2= \ve \alpha_1 $ mod 8 and $4\mid \frac{1}{2}(1-\ve)v(1)-\frac{1}{2}(1+\ve)u(1)+h_1(1)-k_1(1). $
So
$$ \lambda_1\lambda_2 \equiv  (\ve \alpha_2 +(1+\ve)u(1)+2k_1(1) )\ve \alpha_1 \equiv \alpha_1\alpha_2 + (1+\ve)u(1)+2k_1(1)  \mod 4, $$
and $m_3/8\equiv m_1/8$ mod 4, and $m_1m_3/2^6 \equiv 1$ mod 8, iff $2\mid \frac{1}{2}(1+\ve) u(1)+k_1(1)$.
Clearly $2\mid \frac{1}{2}(1+\ve) u(1)+k_1(1)$ iff $2\mid \frac{1}{2}(1+\ve) u(-1)+k_1(-1)$,
giving $m_1m_3/2^6\equiv m_2m_4/2^6$ mod 8, and  $m=m_1m_2m_3m_4/2^{12}\equiv 1$ mod 8.

Similarly, when  $\delta =\ve$ we have
\begin{align*}
\lambda_1 & =\frac{1}{2}(v(1)-u(1))+h_1(1)-k_1(1),\\
\lambda_2 & =\ve+v(1)+u(1)+2h_1(1)+2k_1(1).
\end{align*}
Since
$$ \lambda_1 = \ve \alpha_2 + \frac{1}{2}(1-\ve)v(1) -\frac{1}{2}(1-\ve) u(1) +h_1(1)-k_1(1), $$
we have $2\mid \frac{1}{2}(1-\ve)v(1) -\frac{1}{2}(1-\ve) u(1) +h_1(1)-k_1(1), $ and
$$ \lambda_2 = \ve \alpha_1 +(1-\ve)v(1)+(1-\ve)u(1)+2h_1(1)+2k_1(1) \equiv  \ve \alpha_1 \mod 4. $$
So $\lambda_2 \equiv  \ve \alpha_1 \mod 8, $ $4\mid \frac{1}{2}(1-\ve)v(1)+\frac{1}{2}(1-\ve)u(1)+h_1(1)+k_1(1)$  and
$$\lambda_1\lambda_2 \equiv  (\ve \alpha_2 -(1-\ve)u(1)-2k_1(1) )\ve \alpha_1 \equiv \alpha_1\alpha_2 - (1-\ve)u(1)-2k_1(1)  \mod 4, $$
giving $m_1m_3/2^6 \equiv 1$ mod 8, iff $2\mid \frac{1}{2}(1-\ve) u(1)+k_1(1)$.
Again $m\equiv 1$ mod 8.

\vskip0.1in
\noindent
{\bf (ii) Suppose that $u(1)+v(1)$ is odd.} In this case 
$$ \alpha_1=v(1)-u(1),\;\;\; \alpha_2=\frac{1}{2}(1+u(1)+v(1)). $$
When $\delta =-\ve$ we have 
\begin{align*}
\lambda_1= & \frac{1}{2}(\ve + v(1)-u(1)) +h_1(1)-k_1(1),\\
\lambda_2 = & v(1)+u(1) +2h_1(1)+2k_1(1).
\end{align*}
So
$$ \lambda_1 =\ve \alpha_2 +\frac{1}{2}(1-\ve)v(1) -\frac{1}{2}(1+\ve)u(1)  +h_1(1)-k_1(1), $$
and $2\mid \frac{1}{2}(1-\ve)v(1) -\frac{1}{2}(1+\ve) u(1) +h_1(1)-k_1(1)$, giving
$$ \lambda_2=\ve \alpha_1  + (1-\ve)v(1)+(1+\ve)u(1) +2h_1(1)+2k_1(1) \equiv \ve \alpha_1  \mod 4.$$
Hence  $\lambda_2\equiv \ve \alpha_1$ mod 8,  $4\mid \frac{1}{2}(1-\ve)v(1)+\frac{1}{2}(1+\ve)u(1) +h_1(1)+k_1(1)$ and
$$   \lambda_1\lambda_2 \equiv  (\ve \alpha_2 -(1+\ve)u(1)-2k_1(1) )\ve \alpha_1 \equiv \alpha_1\alpha_2 - (1+\ve)u(1)-2k_1(1)  \mod 4. $$
Thus $m_1m_3/2^6 \equiv 1$ mod 8 iff $2\mid \frac{1}{2}(1+\ve)u(1)+k_1(1).$ Again $m\equiv 1$ mod 8.

When $\delta =\ve$ we have 
\begin{align*}
\lambda_1= & \frac{1}{2}(\ve + v(1)+u(1)) +h_1(1)+k_1(1),\\
\lambda_2 = & v(1)-u(1) +2h_1(1)-2k_1(1).
\end{align*}
Hence
$$ \lambda_1=\ve \alpha_2 + \frac{1}{2}(1-\ve) v(1) +\frac{1}{2}(1-\ve)u(1)+h_1(1)+k_1(1), $$
and $2\mid  \frac{1}{2}(1-\ve) v(1) +\frac{1}{2}(1-\ve)u(1)+h_1(1)+k_1(1),$ giving
$$ \lambda_2= \ve \alpha_1 -(1-\ve) u(1)+(1-\ve)v(1)+2h_1(1)-2k_1(1) \equiv \ve \alpha_1\mod 4. $$
So  $\lambda_2\equiv \ve \alpha_1$ mod 8, $4\mid  \frac{1}{2}(1-\ve) v(1)-\frac{1}{2}(1-\ve)u(1)+h_1(1)-k_1(1)$, and
$$    \lambda_1\lambda_2 \equiv  (\ve \alpha_2 +(1-\ve)u(1)+2k_1(1) )\ve \alpha_1 \equiv \alpha_1\alpha_2 + (1-\ve)u(1)+2k_1(1)  \mod 4. $$
Hence  $m_1m_3/2^6 \equiv 1$ mod 8 iff $2\mid \frac{1}{2}(1-\ve)u(1)+k_1(1).$ Again $m\equiv 1$ mod 8.

\end{proof}


\begin{thebibliography}{99}
\bibitem{dihedral}
T. Boerkoel \& C. Pinner, \textit{Minimal  group determinants and the  Lind-Lehmer problem for dihedral groups},   Acta Arith. \textbf{186} (2018), no. 4, 377-395.	arXiv:1802.07336 [math.NT].


\bibitem{Conrad}
K.\ Conrad, \textit{The origin of representation theory},  Enseign. Math. (2) \textbf{44} (1998), no. 3-4, 361-392.

\bibitem{Frob} F.\ G.\ Frobenius, \textit{\"{U}ber die  Primefactoren  der  Gruppendeterminante},  Gesammelte Ahhand-lungen, Band III, Springer, New York, 1968, pp. 38–77. MR0235974

\bibitem{book}
K.\ Johnson, Group Matrices Group Determinants and Representation Theory, Lecture Notes in Mathematics 2233,  Springer 2019.


\bibitem{Mike}
M.\ Mossinghoff and C.\ Pinner, \textit{Prime power order circulant determinants}, (arxiv 2205.12439v2.)


\bibitem{Newman1}
M.\ Newman, \textit{ On a problem suggested by Olga Taussky-Todd}, Ill. J. Math. \textbf{24} (1980), 156-158.

\bibitem{Newman2}
M.\ Newman, \textit{Determinants of circulants of prime power order}, Linear Multilinear Algebra \textbf{ 9}
(1980), no. 3, 187–191. MR0601702.


\bibitem{bishnu1}
B.\ Paudel and C.\ Pinner, \textit{Integer circulant determinants of order 15,} Integers \textbf{22} (2022), Paper No. A4.

\bibitem{dilum1}
D.\ DeSilva  and C.\ G.\ Pinner, \textit{The Lind-Lehmer constant for $\mathbb Z_p^n$.} Proc. Amer. Math. Soc. \textbf{142} (2014) 1935-1941. 

\bibitem{dilum2}
D.\ DeSilva, M.\ Mossinghoff, V.\ Pigno and  C.\ Pinner, \textit{The Lind-Lehmer constant for certain p-groups.} Math. Comp. \textbf{88} (2019), no. 316, 949-972.


\bibitem{Laquer}
H. T. Laquer, \textit{ Values of circulants with integer entries,}  in A Collection of Manuscripts Related
to the Fibonacci Sequence, Fibonacci Assoc., Santa Clara, 1980, pp. 212–217. MR0624127.

\bibitem{smallgps}
C. Pinner and C. Smyth, \textit{Integer group determinants for small groups},  Ramanujan J. \textbf{51} (2020), no. 2, 421-453.

\bibitem{TausskyTodd}
O. Taussky Todd, \textit{Integral group matrices}, Notices Amer. Math. Soc. 24 (1977), no. 3, A-345.
Abstract no. 746-A15, 746th Meeting, Hayward, CA, Apr. 22–23, 1977.

\bibitem{Yamaguchi0}
Y.\ Yamaguchi and N.\ Yamaguchi, 
\textit{Remark on Laquer's theorem for circulant determinants}, Int. J.  Group Theory 2022, arXiv:2202.06952 [math.RT].

\bibitem{Yamaguchi3}
Y.\ Yamaguchi and N.\ Yamaguchi, 
\textit{Generalized Dedekind’s theorem and its application to integer group determinants}, 2022. arXiv:2203.14420v2 [math.RT].


 
\bibitem{Yamaguchi1}
Y.\ Yamaguchi and N.\ Yamaguchi, 
\textit{Integer circulant determinants of order 16},  Ramanujan J., arXiv:2204.05014 [math.NT].

\bibitem{Yamaguchi2}
Y.\ Yamaguchi and N.\ Yamaguchi, 
\textit{Integer group determinants for $C_2^4$}, 2022.
arXiv:2209.12446v2 [math.NT]


\bibitem{Yamaguchi4}
Y.\ Yamaguchi and N.\ Yamaguchi, \textit{Integer group determinants for $C_4^2$}, arXiv:2211.01597 [math.NT].

\bibitem{Yamaguchi5}
Y.\ Yamaguchi and N.\ Yamaguchi, \textit{Integer group determinants for abelian groups of order 16}, arXiv:2211.14761 [math.NT].
\end{thebibliography}
\end{document}